\documentclass[]{amsart}
\usepackage{amssymb}
\usepackage{amsfonts}
\usepackage{amsmath}
\usepackage{amsthm}
\usepackage{remark}
\usepackage{enumerate}
\usepackage{xypic}
\usepackage{mathrsfs}

\newtheorem{thm}{Theorem}[section]
\newtheorem{theorem}[thm]{Theorem}
\newtheorem{lemma}[thm]{Lemma}
\newtheorem{corollary}[thm]{Corollary}
\newtheorem{proposition}[thm]{Proposition}
\newremark{definition}[thm]{Definition}
\newremark{remark}[thm]{Remark}
\newremark{example}[thm]{Example}

\newcommand{\cO}{\mathcal O}
\newcommand{\cC}{\mathscr{C}}

\newcommand{\cZ}{\mathscr{Z}}
\newcommand{\LL}{\mathcal L}
\newcommand{\Spec}{{\rm Spec}\kern 2pt}
\newcommand{\Proj}{{\rm Proj}\kern 2pt}
\newcommand{\m}{\mathfrak m}
\newcommand{\Frac}{\mathrm{Frac}}
\newcommand{\Pic}{\mathrm{Pic}}
\newcommand{\Z}{\mathbb Z}
\newcommand{\PP}{\mathbb P}
\newcommand{\Q}{\mathbb Q}

\begin{document}

\title{\bf Global Weierstrass equations of hyperelliptic curves}
\author{Qing Liu}

\address{Universit\'e de Bordeaux, Institut de Math\'ematiques de 
Bordeaux, CNRS UMR 5251, 33405 Talence, France} 
\email{Qing.Liu@math.u-bordeaux.fr}

\subjclass[2020]{11G30, 11G05, 14D10, 14H25} 

\thanks{I would like to thank heartily Mohammad Sadek for asking the question
on the existence of a global minimal equation and for stimulating
discussions. I would like to thank Yuri Bilu, Jean-Fran\c cois Jaulent and
Jean Gillibert for helpful discussions about
Lemma~\ref{2-tor}. {I am grateful to one of the anonymous referees for
checking thoroughly the first version, finding mistakes and proposing
several important improvements, especially in the even genera case.}} 

\keywords{Hyperelliptic curves, discriminant, Weierstrass equation,
  good reduction}
 
\begin{abstract} Given a hyperelliptic curve $C$ of genus $g$ over a number field $K$ and
  a Weierstrass model $\cC$ of $C$ over the ring of integers $\cO_K$
  ({\it i.e.} the hyperelliptic involution of $C$ extends to $\cC$ and
  the quotient is a smooth model of $\PP^1_K$ over $\cO_K$),  
  we give necessary and sometimes sufficient conditions for $\cC$ to
  be defined by a global Weierstrass equation. In particular, if $C$
  has everywhere good reduction, we prove that it is defined by a
  global integral Weierstrass equation with invertible discriminant if the
  class number $h_K$ is prime to $2(2g+1)$, confirming a
  conjecture of M. Sadek. 
\end{abstract}
\maketitle

Let $C$ be a hyperelliptic curve of genus $g\ge 2$ or an elliptic curve
over a number field $K$. It is natural to try to represent $C$ with
an integral Weierstrass equation ({\it i.e.} with coefficients in the ring
of integers $\cO_K$ of $K$)
\begin{equation}
  \label{eq:integral-weq}
y^2+Q(x)y=P(x)   
\end{equation}
in a canonical way. For instance we may require the discriminant
of the equation to be \emph{minimal} in the sense that, at
any finite place $v$ of $K$, it divides the discriminant of any Weierstrass
equation of $C$ with coefficients in the local ring $\cO_{v}$
(Definition~\ref{minimal-df}). 
If $C$ has everywhere good reduction, this means that we are looking
for an integral Weierstrass equation with invertible discriminant. 
It is well known that such a minimal equation exists if $\cO_K$ is
principal, but this is not true in general (see e.g. \cite{Sil-84},
Corollary 6 or \cite{Stroeker}, Theorem 1.7, for elliptic curves).
     
The aim of this work is to study general conditions under which a
Weierstrass model, not necessarily minimal, is defined by  
an integral Weierstrass equation (Definition~\ref{integral-eq}). Let us describe
briefly the main results. We work more generally over a Dedekind
domain $R$ with field of fractions $K$. 
For a Weierstrass model $\cC$ of $C$, the quotient $\cZ$ of $\cC$ by the
hyperelliptic involution is a projective bundle over $S=\Spec R$.
For $\cC$ to be defined by an integral Weierstrass equation, there are
three necessary conditions, see Proposition~\ref{obvious}: 
\begin{enumerate}[{\rm (a)}] 
\item $\cZ$ must be trivial ({\it i.e.} $\cZ \simeq \PP^1_S$); 
\item the discriminant ideal $\Delta_{\cC} \subseteq R$ must
  be principal; 
\item if $\omega_{\cC/S}$ denotes the dualizing sheaf of
$\cC$ over $S$, then $\det H^0(\cC, \omega_{\cC/S})\simeq R$.
\end{enumerate}
The converse is true when $g$ is {{not divisible by $4$}} (see below), but is false 
in general (Proposition~\ref{examples-s}(d)).  
These three 
conditions are independent of each other (see \S~\ref{examples})
except when $g=1, 2$ or $4$ in which case $\Delta_{\cC}$ is
isomorphic to some power of
$\det H^0(\cC, \omega_{\cC/S})$ (Corollary~\ref{D-O}). 
\medskip

Roughly speaking, Condition (a) means there is a rational function $x$
for the integral Weierstrass equation \eqref{eq:integral-weq} potentially
representing $\cC$. This condition is automatically satisfied if
$\Pic(S)$ is finite of odd order (Proposition~\ref{prop:P1}) or if $\cC$ is a pointed
Weierstrass model (Proposition~\ref{pointed-main}) 
The existence of $y$ for the equation means that
the integral closure of $R[x]$ in the function fields $K(C)$ is
free over $R[x]$, with a basis $\{1, y\}$.
We define the \emph{Weierstrass class}
$[\mathfrak w]\in \Pic(S)$ associated to $\cC$ (Definition~\ref{w-c}) and
we prove that the existence of an integral
Weierstrass equation representing $\cC$ is equivalent to
the triviality of $[\mathfrak w]$ and {{$\cZ\simeq \PP^1_S$}} (Proposition~\ref{global-L}). 
\medskip

The main result of this work is the following theorem. Denote by
$f : \cC\to S$ the structure morphism. 

\begin{theorem}[See Theorem \ref{main}] 
 Let $\cC$ be a Weierstrass model of $C$ over $S$.  Then $\cC$ is
 defined by an integral Weierstrass equation under any of the following
 conditions: 
  \begin{enumerate}[{\rm (1)}] 
  \item $\Delta_{\cC}$ is principal, $\det f_*\omega_{\cC/S}$ is free and either 
    \begin{enumerate}
    \item $g$ is odd,  $\cZ\simeq \PP^1_S$ or 
    \item $g\equiv 2 \mod 4$; 
    \end{enumerate}
  \item $\Pic(S)$ is finite of odd order, 
    $\Delta_{\cC}$ is principal and 
    $\det f_*\omega_{\cC/S}$ is free; 
\item $\Pic(S)$ is finite of order prime to $2(2g+1)$ 
 and $\Delta_{\cC}$ is principal;
\item     $\Pic(S)$ is finite of order prime to $2g$  
and $\det f_*\omega_{\cC/S}$ is free. 
  \end{enumerate}
\end{theorem}

In particular, for hyperelliptic curves having everywhere good
reduction over a number field, the next corollary  
answers positively a conjecture of Mohammad Sadek~\cite{Sad}.  It is
well known for elliptic curves. 

\begin{corollary}[See Corollary \ref{conj-Sadek}] Let
  $C$ be a hyperelliptic curve of genus $g$ over a number field $K$ having everywhere
  good reduction. If  the class number of $K$ is prime to $2(2g+1)$,
  then  $C$ is defined by an integral Weierstrass 
  equation with invertible discriminant. 
\end{corollary}

Similar results are given in \S \ref{p-h} for pointed hyperelliptic curves
$(C, P)$ where $P$ is a rational Weierstrass point of $C$. Finally in
\S \ref{examples}, we construct  examples of $\cC$ satisfying various nice
properties but
not defined by integral (or pointed integral) Weierstrass equations. 
\medskip 

\noindent{\bf Settings} In all this paper, $R$ is a Dedekind domain,
$K=\Frac(R)$, $S=\Spec R$ and $C$ is a hyperelliptic curve over $K$ 
(that is, $C$ is smooth projective geometrically connected over $K$ and admits
a finite separable morphism $C\to \mathbb P^1_K$ of degree $2$) of genus $g\ge 2$
or an elliptic curve over $K$. 

\begin{section}{Twisted projective lines} \label{TPL} 

Let $\phi : \cZ\to S$ be a smooth projective scheme with generic fiber
isomorphic to  $\PP^1_K$.  
This is a \emph{twisted $\PP^1_S$} (see \cite{LK}, 3.2). It is the projective
bundle over $S$ defined by a locally free rank $2$ vector bundle on $S$.
\medskip 

{{In this section we study the sections of $\cZ$ and
also give a sufficient condition for $\cZ$ to be isomorphic to $\PP^1_S$.
We often consider sections of $\cZ$ as effective Cartier divisors on $\cZ$. Then
$\cO_{\cZ}(D)$ is a subsheaf of the function field $K(\cZ)$.}} 
\medskip

\begin{definition} \label{def:es}
  {Let $D\in \cZ(S)$ be a section. 
We will say that $D$ is \emph{elementary} 
if $\phi_*\cO_{\cZ}(D)$ is free over $\cO_S$.}
\end{definition}

\begin{lemma} \label{element-s} Let $D\in \cZ(S)$. Let $F=\phi_*\cO_{\cZ}(D)$.
  \begin{enumerate}[{\rm (1)}]
  \item The invertible sheaf $\cO_{\cZ}(D)$ is very ample, $F$ is locally free of rank $2$ over $\cO_S$, the morphism
    $i : \cZ\to \PP(F)$ associated to $\cO_{\cZ}(D)$ is an isomorphism and we have 
\[
i^*\cO_{\PP(F)}(1) \simeq \cO_{\cZ}(D). 
\]
  \item The following properties are equivalent:
    \begin{enumerate}
    \item $D$ is elementary;
    \item $\det F\simeq \cO_S$; 
      \item There exists an isomorphism $\theta : \cZ\to \PP^1_S$ such
        that  $\cO_{\cZ}(D)\simeq \theta^*\cO_{\PP^1_S}(1)$. 
  \end{enumerate}
\item Let $D'\in \cZ(S)$.  Then
    there exists a unique $M\in \Pic(S)$ such that
    $$\cO_{\cZ}(D')\simeq \cO_{\cZ}(D)\otimes \phi^*M.$$ 
    We then have
    \[
    \phi_*\cO_{\cZ}(D')\simeq F\otimes M, \quad
    \det \phi_*\cO_{\cZ}(D')\simeq (\det F)\otimes M^{\otimes 2}.
    \]  
\item For any $M\in \Pic(S)$, there is a section
  $D'\in \cZ(S)$ such that
  $$\cO_{\cZ}(D')\simeq \cO_{\cZ}(D)\otimes \phi^*M.$$  
  \end{enumerate}
\end{lemma}

\begin{proof} (1) This is true because it is true locally on $S$. 
\medskip 

(2) Obviously (c) $\Rightarrow$ (a) $\Rightarrow$ (b).
We have (b) $\Rightarrow$ (a) because $F\simeq \cO_S\oplus \det F$, 
$R$ being a Dedekind domain. 

(a) $\Rightarrow$ (c). Compose the 
isomorphism $i: \cZ\simeq \PP(F)$ with an isomorphism $\PP(F)\simeq \PP^1_S$ to get
the desired isomorphism $\theta$. Then $\theta^*\cO_{\PP^1_S}(1)\simeq \cO_{\cZ}(D)$ by the isomorphism of Part (1). 
\medskip

(3) For any noetherian connected scheme $T$ and any locally free
$\cO_T$-module $G$ of finite rank, it is known that there is a
canonical isomorphism
\begin{equation} \label{eq:picP} 
\Pic(\PP(G))\simeq \langle \cO_{\PP(G)}(1)\rangle \times
  \Pic(T)
\end{equation}
(see \cite{EGA}, II, Remarque 4.2.7, or \cite{GLL},
Proposition 8.4(a)). This implies the existence and uniqueness of $M$.   
The second part comes from the projection formula (\cite{L_B}, Proposition
5.2.32) and the fact that $F$ is locally free of rank $2$.
\medskip

(4) We work over $\PP(F)$ with $\cO_{\PP(F)}(1)\simeq\cO_{\PP(F)}(D)$.
Let $H=M^{\vee}$ be the dual of $M$. Using  
the structure of the projective modules over Dedekind domains,
$H$ can be injected into $F$ with
locally free quotient of rank 1. Then $H$ induces a section $D'$ of
$\PP(F)$ 
(\cite{EGA}, II.4.2.3). As an effective Cartier divisor on $\PP(F)$,
$D'$ is defined by the sheaf of homogeneous ideals $H\cO_{\PP(F)}$. 
One can check that $\cO_{\PP(F)}(-D')\otimes \cO_{\PP(F)}(1)\simeq\phi^*H$, thus 
the desired isomorphism. 
\end{proof}

\begin{example}
  \begin{enumerate}[{\rm (1)}] 
    \item If $R$ is a PID, then all sections of $\cZ$ are elementary. 
\item 
  {{On $\PP^1_S$ with a parameter $x$, a typical elementary section is the pole divisor $\mathrm{div}_{\infty}(x)$ of $x$ with
      $\cO_{\PP^1_S}(\mathrm{div}_{\infty}(x))\simeq \cO_{\PP^1_S}(1)$.}

   In view of Lemma~\ref{element-s}(2), when
  $\cZ$ has an elementary section $D$, we will often think about the pair 
$(\cZ, D)$ as $(\PP^1_S, \mathrm{div}_{\infty}(x))$ where $x$ is some parameter of $\PP^1_S$. }
\end{enumerate}
\end{example}

\begin{proposition}  \label{prop:P1}
We have $\cZ\simeq \PP^1_S$ if and only if for some 
$D\in \cZ(S)$, $\det\phi_*\cO_{\cZ}(D)$ is a square in $\Pic(S)$. 
In particular, if the map
$$\Pic(S)\to \Pic(S), \quad x\mapsto x^2$$ 
is surjective (e.g. if $\Pic(S)$ has finite and odd order), then
$\cZ\simeq \PP^1_S$.  
\end{proposition}

\begin{proof} Let $D\in \cZ(S)$ and $F=\phi_*\cO_{\cZ}(D)$. We have
  $\cZ\simeq \PP(F)$ by Lemma~\ref{element-s}(1). 
  By Lemma~\ref{lem:iso-pb} below, 
$\cZ\simeq \PP^1_S$ if and only if $F\simeq M\oplus M$ for some
$M\in \Pic(S)$. The existence of such an isomorphism
implies that $\det F \simeq M^{\otimes 2}$ is a
square in $\Pic(S)$. Conversely, if $\det F\simeq M^{\otimes 2}$ for
some $M\in \Pic(S)$, 
replacing $F$ by $F\otimes M^{\vee}$ (which does not change
$\PP(F)$), we can suppose that $\det F \simeq \cO_S$. 
Therefore $F$ is free (Lemma~\ref{element-s}(2)) and $\cZ\simeq \PP^1_S$. 
\end{proof}

\begin{lemma} \label{lem:iso-pb} 
Let $F, G$ be two locally free finite rank sheaves on a noetherian scheme $T$. 
Then the projective bundles $\PP(F)$ and $\PP(G)$ are $T$-isomorphic if
and
only if $F\simeq G\otimes_{\cO_T} M$ for some $M\in \Pic(T)$.    
\end{lemma}

\begin{proof} The if part is a general fact (\cite{EGA}, II.4.1.4). To prove
  the converse, we can suppose that $T$ is connected. We have a
  commutative diagram 
$$ 
\xymatrix{
\PP(F) \ar[rr]^{\sigma} \ar[rd]_{q} &  & \PP(G) \ar[ld]^{p} \\ 
& T & }
$$
with an isomorphism $\sigma$. Using the structure of $\Pic(\PP(G))$ 
(see Isomorphism \eqref{eq:picP} in the proof of Lemma~\ref{element-s}(3)), 
there exists a $M\in \Pic(T)$ such that 
$\sigma_*\cO_{\PP(F)}(1)  = \cO_{\PP(G)}(1)\otimes p^*M$. Taking 
$p_*$ and using the projection formula (\cite{L_B}, Proposition
5.2.32),
we find 
$$q_* \cO_{\PP(F)}(1) = p_*\cO_{\PP(G)}(1)\otimes M.$$ 
But $q_* \cO_{\PP(F)}(1) \simeq F$ (\cite{SP}, Example 27.21.3) and 
$p_*\cO_{\PP(G)}(1)\simeq G$, the only if part is proved. 
\end{proof}

\end{section}

\begin{section}{Weierstrass models and discriminant ideals}

  In this section we gather some basic definitions and facts on Weierstrass
  equations and Weierstrass models. 
  \medskip

A \emph{model of $C$ over $S$} is a normal projective scheme
$f: \cC \to S$
whose generic fiber $\cC_K$ is isomorphic to $C$. 
A model $\cC$ is a \emph{Weierstrass model} if the hyperelliptic
involution on $C$ extends to an $S$-involution
$\sigma$ on $\cC$ and if the quotient
$$\cZ:=\cC/\langle \sigma \rangle$$
is smooth over $S$. The latter is then a smooth model of $\PP^1_K$ over
$S$. Equivalently, a Weierstrass model is the
normalization of a smooth model $\cZ$ of $\PP^1_K$ in the
canonical degree $2$ morphism $C\to \PP^1_K$. 
\medskip 

\begin{definition}\label{integral-eq} 
A \emph{Weierstrass equation of $C$} is 
an equation of the form
\begin{equation}
  \label{eq:def-Weq}
    y^2+Q(x)y=P(x), \quad P(x), Q(x)\in K[x]
\end{equation}
with $\deg Q(x)\le g+1, \deg P(x) \le 2g+2$ such that the
corresponding affine curve is isomorphic to an open subscheme of $C$. 
Such an equation is said to be \emph{integral over $R$} if moreover $P(x),
Q(x)\in R[x]$. If an integral Weierstrass equation \eqref{eq:def-Weq}
of $C$ defines a normal affine scheme 
\begin{equation} \label{eq:aff-W} 
\Spec R[x,y]/(y^2+Q(x)y-P(x)) 
\end{equation}
(see \cite{L_Tr}, \S 3, Lemme 2 for a criterion of normality), then
embedding $\Spec R[x]$ in $\PP^1_S$ and taking the
normalization of $\PP^1_S$ in the degree $2$ morphism $C\to \PP^1_K$
will give a Weierstrass model $\cC$. Such a Weierstrass model is said
to be \emph{defined by the integral Weierstrass equation} \eqref{eq:def-Weq}. 
{The scheme $\cC$ is then obtained by glueing the two affine schemes: 
$$\Spec R[x,y]/(y^2+Q(x)y-P(x)), \  
\Spec R[t,z]/(z^2+t^{g+1}Q(1/t)z-t^{2g+2}P(1/t))$$
via the identifications $t=1/x$, $z=y/x^{g+1}$.}
\end{definition}

\begin{remark}
Roughly speaking, for $\cC$ to be defined by an integral Weierstrass equation, there
are two conditions: $\cZ\simeq \PP^1_S$ (there exists an $x$) and the
quotient by $R[x]$ of the integral closure of $R[x]$ in $K(C)$ is free
(there exists an $y$).
\end{remark}

\begin{lemma} \label{loc-W} Suppose that $R$ is a PID.   Then all 
  Weierstrass models are defined by integral Weierstrass equations. 
\end{lemma}

\begin{proof}
This is well known. One can for instance argue as follows: by 
Proposition~\ref{prop:P1}, $\cZ\simeq \PP^1_S$. Choose any
(automatically elementary) section $D\in \cZ(S)$ and 
{{argue as in the proof of}} Proposition~\ref{global-L}.  
\end{proof}

\begin{definition} \label{minimal-df} 
  Now we define the discriminant ideal and minimal Weierstrass models.
Let $\cC$ be a Weierstrass model over $S$. For
any closed point $s\in S$, $\cC\times_S \Spec \cO_{S,s}$ is then
defined by a Weierstrass equation over the PID $\cO_{S,s}$. The
discriminant of the equation {{(\cite{L_Tr}, \S 2)}} generates an ideal $\Delta_s$ of
$\cO_{S,s}$
independent on the choice of the equation. We have $\Delta_s=\cO_{S,s}$ if
and only if $\cC$ is smooth at $s$.  
The \emph{discriminant ideal $\Delta_{\cC}$ of $\cC$} is the ideal
$\prod_s \Delta_s$. We have $\Delta_\cC=\cO_S$ if and only if $\cC$ is
smooth over $S$ (then $C$ has everywhere good reduction). 
We say that $\cC$ is \emph{a minimal Weierstrass model}
if for all $s$, $\Delta_s$ divides the discriminant of any integral
Weierstrass equation of $C$ over $\cO_{S,s}$. 
\end{definition}

\begin{remark}\label{minimal-ex} The curve $C$ always has minimal
  Weierstrass models. Indeed, $C$ has a smooth projective model
  over a dense open subset $U$ of $S$. Then one can glue this model
  with minimal Weierstrass models at the points of $S\setminus U$. 
Note that in general the minimal Weierstrass models are not unique
(unless $C$ has good reduction everywhere). However, up to
$S$-isomorphisms, there are only finitely many such 
models for a given $C$ (use \cite{L_Tr}, \S 8, Corollary 4). 
\end{remark}

Given a Weierstrass model $\cC$, there are some natural conditions for
$\cC$ to be defined by an integral Weierstrass equation. As
$f : \cC\to S$ is locally a complete intersection, its
dualizing sheaf $\omega_{\cC/S}$ is invertible and $f_*\omega_{\cC/S}$ is
locally free of rank $g$ over $S$.

\begin{proposition} \label{obvious}
  If a Weierstrass model $\cC$ is defined by an integral
  Weierstrass equation, then
\begin{enumerate}[{\rm (a)}] 
  \item  $\cZ \simeq \PP^1_S$;
  \item $\Delta_{\cC}$ is a principal ideal of $R$; 
  \item $\det f_*\omega_{\cC/S}$ is free.  
\end{enumerate}
\end{proposition}

\begin{proof} The first two conditions are clear. If $\cC$ is defined
  by an integral Weierstrass equation~\eqref{eq:def-Weq}, then one can
  check exactly as over a ground field (see {\it e.g.}  \cite{L_B}, Proposition
  7.4.26) that $f_*\omega_{\cC/S}$ is free over $R$ with 
  $$\dfrac{x^idx}{2y+Q(x)}, \quad 0\le i \le g-1$$ 
  as a basis. So $\det f_*\omega_{\cC/S}$ is free.  
\end{proof}
\end{section}

\begin{section}{Weierstrass class of a Weierstrass model}
  
We fix a Weierstrass model $\cC$ and we look for conditions ensuring
it is defined by an integral Weierstrass equation. 
Recall that $f: \cC\to S$ is the structure morphism. Denote by 
$\pi: \cC \to \cZ$ the quotient morphism by the hyperelliptic
involution and by $\phi: \cZ\to S$ the structure morphism:
$$ 
\xymatrix{
\cC \ar[rr]^{\pi} \ar[rd]_{f} &  & \cZ \ar[ld]^{\phi} \\ 
& S. & }
$$
The
morphism $\pi$ is finite, locally free of degree $2$, 
and we have an exact sequence  
\begin{equation}
  \label{eq:L}
  0 \to \cO_{\cZ} \to \pi_*\cO_{\cC} \to \LL:=\pi_*\cO_{\cC}/\cO_{\cZ} \to 0. 
\end{equation}

\begin{lemma} \label{lem:L}Let $D\in \cZ(S)$ be a section of $\cZ$ over $S$.
Let $\LL$ be as above. 
  Then there exists a unique $N\in \Pic(S)$, depending on $D$, such that
\begin{equation} \label{eq:LN} 
\LL \otimes \cO_{\cZ}((g+1)D)\simeq \phi^*N.
\end{equation}
\end{lemma} 

\begin{proof} For any $s\in S$, we let $U\ni s$ be an affine open
  neighborhood such that $\cC_U$ is defined by an integral Weierstrass equation
  (Lemma~\ref{loc-W})
  $$ y^2_U+Q_U(x_U)y_U=P_U(x_U)$$
  (thus $\cZ_U\simeq \PP^1_U$ with parameter $x_U$) {{and that $D|_{\cZ_U}$ is the pole divisor of  $x_U$.}} 
  Consider the covering of $\cZ_U$ by
  $$W:=\Spec  \cO_{S}(U)[x_U], \quad V:=\Spec \cO_{S}(U)[1/x_U].$$  
Then 
$$(\pi_*\cO_{\cC})(W)=\cO_{\cZ}(W) \oplus \cO_{\cZ}(W)y_U 
=\cO_{\cZ}(W) \oplus \cO_{\cZ}(-(g+1)D)(W)y_U$$
and  
$$(\pi_*\cO_{\cC})(V)=\cO_{\cZ}(V) \oplus  \cO_{\cZ}(V)(y_U/x_U^{g+1})
=\cO_{\cZ}(V) \oplus \cO_{\cZ}(-(g+1)D)(V)y_U.$$
Therefore 
$$\pi_*\cO_{\cC}|_{\cZ_U}=\cO_{\cZ_U} \oplus  \cO_{\cZ_U}(-(g+1)D)y_U$$
and 
$$\LL|_{\cZ_U} \simeq \cO_{\cZ}(-(g+1)D)|_{\cZ_U}.$$
This proves the existence of $N$. The uniqueness of $N$ up to isomorphism
comes from the fact that $N\simeq \phi_*\phi^* N$. 
\end{proof}

\begin{remark}[Dependence on $D$]  \label{change-D} Denote by $F=\phi_*\cO_{\cZ}(D)$. 
 Let $D'\in \cZ(S)$ be another section with $F'=\phi_*\cO_{\cZ}(D')$
 and corresponding $N'\in \Pic(S)$ as in   Lemma~\ref{lem:L}. Then by
  Lemma~\ref{element-s}(3), 
  $\cO_{\cZ}(D')\simeq \cO_{\cZ}(D)\otimes \phi^* M$ for a unique $M\in \Pic(S)$
and we have 
$$ N'\simeq N\otimes M^{\otimes (g+1)}, \quad 
\det F' \simeq (\det F)\otimes M^{\otimes 2}.$$
In particular, 
    $$G_{\cC,2}:=N^{\otimes 2} \otimes (\det F^{\vee})^{\otimes
      (g+1)}\in \Pic(S)$$
is independent of the choice of $D$. 
If $g$ is odd, then $G_{\cC,2}$ is the square of  
$$G_{\cC,1}:=N\otimes (\det F^{\vee})^{\otimes (g+1)/2}\in \Pic(S)$$
and the latter is also independent on the choice of $D$.
\end{remark}

\begin{definition} \label{w-c} Let $\cC$ be a Weierstrass model. 
We define the \emph{Weierstrass class $[\mathfrak w]\in \Pic(S)$ of $\cC$} as the class of
$G_{\cC,1}$ (resp. $G_{\cC,2}$) when $g$ is odd (resp. even). 
\end{definition}

\begin{proposition} \label{global-L}
Let $f : \cC\to S$ be a Weierstrass  model of $C$. Then $\cC$ can 
be  defined by an integral Weierstrass equation if and only if
$\cZ\simeq \PP^1_S$ and $[\mathfrak w]$ is trivial.
\end{proposition}

\begin{proof} {\bf Step (1)} Suppose there exists an elementary
 section $D\in   \cZ(S)$ such that the corresponding $N$ is trivial
  (Lemma~\ref{lem:L}). Let us show the existence of an integral
  Weierstrass equation defining $\cC$.
  
We identify $(\cZ, D)$ with $(\PP^1_S, \mathrm{div}_{\infty}(x))$ for some parameter $x$
  of $\PP^1_S$. Let $y\in K(C)$ be a basis of $\LL\otimes \cO_{\cZ}((g+1)D)$. 
Over $W:=\Spec R[x]$, we have 
$\LL\otimes \cO_{\cZ}((g+1)D)(W)=R[x]y$.
This means that
\begin{equation}
  \label{eq:mim}
  \cO_{\cC}(\pi^{-1}(W))=R[x]\oplus R[x]y. 
\end{equation}
So there exist $Q(x), P(x) \in R[x]$ such that 
$y^2+Q(x)y=P(x)$. 
It remains to bound the degrees of $Q(x)$ and $P(x)$. 
Over $V:=\Spec R[1/x]$, we have 
$$\LL(V)=\cO_{\cZ}(-(g+1)D)(V)y=R[1/x](y/x^{g+1}),$$
which implies that $y/x^{g+1}\in \cO_{\cC}(\pi^{-1}(V))$, hence
integral over $R[1/x]$. Therefore 
$Q(x)/x^{g+1}, P(x)/x^{2g+2}\in R[1/x]$ and $\deg Q(x)\le g+1$,
$\deg P(x)\le 2g+2$. The pair $x, y$
gives an integral Weierstrass equation for $\cC$.
\medskip

{\bf Step (2), ``only if'' part.} Suppose that $\cC$ is defined by an
integral Weierstrass equation as Equation~\eqref{eq:def-Weq}.
We have $\cZ\simeq \PP^1_S$ with parameter $x$. The computations in the proof of Lemma~\ref{lem:L} show that for 
$D=\mathrm{div}_{\infty}(x)$, the corresponding $N$ is free. As $D$ is
elementary, $[\mathfrak w]$ is trivial. 
\medskip

{\bf Step (3), ``if'' part.}  Suppose that $\cZ\simeq \PP^1_S$
and $[\mathfrak w]$ is trivial. Let $x$ be a parameter of $\PP^1_S$ and let
$D=\mathrm{div}_{\infty}(x)$. If $g$ is odd, then $N\simeq [\mathfrak w]$ is
trivial. It follows from Step (1) that $\cC$ is defined by an integral
Weierstrass equation.

Suppose now that $g$ is even. Then $N^{\otimes 2}\simeq [\mathfrak w]$
is trivial. 
By Lemma~\ref{element-s}(4), there exists a
section $D'\in \cZ(S)$ such that
$$\cO_{\cZ}(D')\simeq \cO_{\cZ}(D)\otimes \phi^*N.$$  
This implies that
$\det\phi_*\cO_{\cZ}(D')\simeq N^{\otimes 2}$ and
$N'\simeq N\otimes N^{\otimes (g+1)}$ are free, therefore $\cC$ is
defined by an integral Weierstrass equation, again by applying Step (1) to
$D'$.
\end{proof}

 \begin{remark} In view of this {{proposition}}, 
    our Remarque 6 in \cite{L_Tr}, page 4584, is very likely to be wrong. 
  \end{remark}

 \begin{remark}[Equations defining the same Weierstrass model]
   Suppose that $\cC$ is defined by an integral Weierstrass equation
 $y^2+Q(x)y=P(x)$. 
\begin{enumerate}[{\rm (a)}]
\item  If $g$ is even or if $\Pic(S)[2]=\{ 1\}$, then any other integral Weierstrass equation $z^2+G(u)z=F(u)$  of $\cC$ is related to the previous one by 
 $$ u=\frac{ax+b}{cx+d}, \quad  \text{with} 
\begin{pmatrix}
  a & b \\
  c & d
\end{pmatrix} \in \mathrm{Gl}_2(R)$$
and 
$$z =\frac{ey + H(x)}{(cx+d)^{g+1}}, \quad e\in R^*, H(x)\in R[x], \deg H(x)\le g+1. $$
Indeed, let $D'=\mathrm{div}_{\infty}(u)$ with corresponding $F'=\phi_*\cO_{\cZ}(D')\simeq \cO_S^2$ and $N'$ as in Lemma~\ref{lem:L}. As seen in Step (2) of the proof of Proposition~\ref{global-L}, we have $N'\simeq \cO_S$. Let $M$ be as in Remark~\ref{change-D}, then
$M^{\otimes 2}\simeq \cO_S$ and $M^{\otimes (g+1)}\simeq \cO_S$. So $M\simeq \cO_S$, and
this implies the relation between $x$ and $u$. The condition on $z$ comes from the
fact that $(cx+d)^{g+1}z\in K(C)$ is integral over $R[x]$ and that both equations
give the same discriminant ideal $\Delta_{\cC}$.
\item 
On the opposite, suppose that
$g$ is odd and that there exists $M\in \Pic(S)$ of order exactly $2$. 
We take a section $D'\in \cZ(S)$ as in Lemma~\ref{element-s}(4). Then $D'$ is
elementary as $M^{\otimes 2}\simeq \cO_S$. Moreover the corresponding
$N'\simeq N\otimes M^{\otimes g+1}\simeq \cO_S$. By Step (1) in the proof of
Proposition~\ref{global-L}, if $D'=\mathrm{div}_\infty(u)$, then $\cC$
is defined by an integral Weierstrass equation $z^2+G(u)z=F(u)$. But
$u, x\in K(\cZ)$ are not related by the action of $\mathrm{Gl}_2(R)$. 

Note that this phenomena does not happen for pointed Weierstrass models
(Definition~\ref{def:p-w}) as we then have $\mathrm{div}_\infty(u)=
\mathrm{div}_\infty(x)$, so $u=ax+b$ for some $a\in R^*$ and $b\in R$. 
\end{enumerate}
 \end{remark}
 
{{Now we relate the Weierstrass class to some other classes
attached to $\cC$.}} 

\begin{lemma}[Approximation of local equations] \label{change-eq} Let $\cC$ be a Weierstrass model of $C$
  over $S$. Let $D\in \cZ(S)$. Let us fix
  \begin{enumerate}[\rm (a)]
  \item a Weierstrass equation of $C$ over $K$:  
    \begin{equation}
      \label{eq:w-g}
    y^2+Q(x)y=P(x)  
    \end{equation}
    with $x$ having
    its pole at $D_K$ and
  \item a non-empty open subset $U\subset S$.
  \end{enumerate}
Then there exist 
$r\in \cO_{S}(U)$, $h(x)\in \cO_S(U)[x]$ of degree $\le g+1$ and,
for all $s\in S\setminus U$, an integral Weierstrass equation
\begin{equation}
  \label{eq:w-s}
y_s^2+Q_s(x_s)y_s=P_s(x_s)   
\end{equation}
for $\cC\times_S \Spec\cO_{S,s}$ with
$$x-r=a_sx_s, \quad y-h(x)=b_sy_s, \quad a_s, b_s\in K^*.$$ 
\end{lemma}

\begin{proof} For each $s\in S\setminus U$, there exists an
integral Weierstrass equation 
$$
y_s^2+Q_s(x_s)y_s=P_s(x_s) 
$$
for   $\cC\times_S \Spec\cO_{S,s}$ with $x_s$ having its pole at 
$D_K$. So there exist $a_s, b_s\in K^*$, $r_s\in K$,
$h_s(x)\in K[x]$ of degree $\le g+1$ such that
$$
x=a_sx_s+r_s, \quad y=b_sy_s+h_s(x). 
$$
By the approximation property in Dedekind domains,
there exist $r\in \cO_{S}(U), h(x)\in \cO_{S}(U)[x]$ of degree $\le g+1$ such that
$$ r-r_s\in a_s\cO_{S,s}, \quad h(x)-h_s(x)\in b_s\cO_{S,s}[x]$$
for all $s\in S\setminus U$. 
We obtain the desired result after replacing
$x_s$ with $x_s-(r-r_s)/a_s$ and $y_s$ with $y_s-(h(x)-h_s(x))/b_s$. 
\end{proof}

\begin{definition} \label{def:ab} Keep the settings of Lemma~\ref{change-eq}. 
Choose an open subset $U$ such that the equation \eqref{eq:w-g} has
good reduction over $U$ and defines $\cC_U$
{{(so $D|_{\cZ_U}$ is the pole divisor of $x|_{\cZ_U}$).}}  
Define $\mathfrak a$ as the fractional ideal of $K$ such that
$\mathfrak a\otimes \cO_{S,s}=a_s\cO_{S,s}$ for all $s\in S\setminus
U$
and $=\cO_{S,s}$ for $s\in U$. Define similarly
$\mathfrak b$. These ideals depend on the choice
of Equation~\eqref{eq:w-g} and on $D$, 
but don't depend on the choice of $U$. Their classes in $\Pic(S)$
don't depend on the choice of Equation~\eqref{eq:w-g} over $K$. 
\end{definition}

\begin{lemma}  \label{a-b-F-N} Let $\cC$ be a Weierstrass model of $C$ over $S$. 
Let $D\in \cZ(S)$. Let $F=\phi_*\cO_{\cZ}(D)$ and let
$N$ be as in Lemma~\ref{lem:L}.
\begin{enumerate}[\rm (1)]
\item  We have 
$$ \det F \simeq \mathfrak a^{-1}, \quad N\simeq \mathfrak b^{-1}$$
and
$$\Delta_{\cC} \simeq (\mathfrak a^{g+1}\mathfrak b^{-2})^{2(2g+1)}, \quad 
\det f_*\omega_{\cC/S} \simeq \mathfrak a ^{-g(g+1)/2} \mathfrak b^g. $$
\item Let $[\mathfrak w]$ be the Weierstrass class of $\cC$
  (Definition~\ref{w-c}). Then 
$$ \Delta_{\cC} \simeq [\mathfrak w]^{\otimes 2(2g+1)},
    \quad \det f_*\omega_{\cC/S} \simeq [\mathfrak w]^{\otimes (-g/2)}$$
if $g$ is even, and
$$ \Delta_{\cC} \simeq   [\mathfrak w]^{\otimes 4(2g+1)}, \quad
\det f_*\omega_{\cC/S} \simeq [\mathfrak w]^{\otimes  (-g)}$$
if $g$ is odd.
\item In $\Pic(S)$, we have the following equalities of subgroups
  $$ \langle [\mathfrak w]^2 \rangle = \langle \Delta_{\cC}, \det f_*\omega_{\cC/S}\rangle$$
  if $4 \mid g$, and
  $$ \langle [\mathfrak w] \rangle = \langle \Delta_{\cC},
  \det f_*\omega_{\cC/S}\rangle$$ 
otherwise. 
\end{enumerate}
\end{lemma} 

\begin{proof} Part (2) follows immediately from (1) and implies Part (3).
Let us prove Part (1). Keep the notation of Lemma~\ref{change-eq}. We choose the
  equation~\eqref{eq:w-g} with coefficients in $\cO_S(U)$ and 
  having good reduction over $U$ for some dense open subset $U\subseteq S$.
  We also translate $x$ and $y$ so that $x=a_sx_s$ and $y=b_s y_s$ at
  $s\in S\setminus U$. 
\medskip

  We have
  $$ \phi_*\cO_{\cZ}(D)(U)=\cO_S(U)\oplus \cO_S(U)x, \quad
  \phi_*\cO_{\cZ}(D)_s=\cO_{S,s}\oplus \cO_{S,s}x_s,$$
  so we have canonically 
  $$ (\det F)(U)=\cO_{S}(U)(1\wedge x), \quad (\det F)_s=\cO_{S,s}(1\wedge x_s).$$
  Then the relations $x=a_sx_s$ imply that $\det F\simeq \mathfrak
  a^{-1}$.
  By the computations in the proof of Lemma~\ref{lem:L}, we see that 
  $N$ has bases $y$ over $\cO_S(U)$ and $y_s$ over $\cO_{S,s}$.
  The relations $y=b_sy_s$ then imply that $N\simeq \mathfrak
  b^{-1}$. 
\medskip 
  
Denote by $\Delta$ the discriminant of Equation~\eqref{eq:w-g}
and by $\Delta_s$ that of Equation~\eqref{eq:w-s}.  
Then
$$ \Delta=b_s^{4(2g+1)} a_s^{-2(g+1)(2g+1)}\Delta_s  $$
(see \cite{L_Tr}, bottom of page 4581). This proves that 
$$\Delta R= \mathfrak b^{4(2g+1)} \mathfrak a^{-2(g+1)(2g+1)}\Delta_{\cC}$$
as fractional ideals of $K$, hence
$$\Delta_{\cC} \simeq (\mathfrak a^{g+1}\mathfrak b^{-2})^{2(2g+1)}. $$
Now we compute $f_*\omega_{\cC/S}$. Denote by
$$\omega_U= 
\frac{dx}{(2y+Q(x))} \wedge ... \wedge 
\frac{x^{g-1}dx}{(2y+Q(x))}$$
and 
$$\omega_s= 
\frac{dx_s}{(2y_s+Q_s(x_s))} \wedge ... \wedge 
\frac{x_s^{g-1}dx_s}{(2y_s+Q(x_s))}.$$
They are respectively bases of $(\det f_*\omega_{\cC/S})(U)$ and
of $(\det f_*\omega_{\cC/S})_s$.
As $x=a_sx_s$ and $y=b_sy_s$, we have $Q_s(x_s)=Q(x)/b_s$. Thus
$$\frac{x^idx}{2y+Q(x)}=(a_s^{i+1}b_s^{-1}) \frac{x_s^idx_s}{2y_s+Q(x_s)}$$ 
and $\omega_U=a_s^{g(g+1)/2}b_s^{-g} \omega_s$, 
so
$ \det f_*\omega_{\cC/S} \simeq
\mathfrak a ^{-g(g+1)/2} \mathfrak b^g$.
\end{proof}

\begin{corollary} \label{D-O}Let $\cC$ be a Weierstrass model of $C$
  over $S$. Then the following properties are true. 
  \begin{enumerate}[{\rm (1)}] 
  \item If $g=1$, then $\Delta_{\cC} \simeq (\det
    f_*\omega_{\cC/S})^{\otimes (-12)}$ (well-known);
\item If $g=2$, then $\Delta_{\cC/S}  \simeq (\det f_*\omega_{\cC/S})^{\otimes (-10)}$.  
 \item If $g=4$, then $\Delta_{\cC/S}  \simeq (\det
   f_*\omega_{\cC/S})^{\otimes (-9)}$. 
\item In $\Pic(S)$, $\Delta_{\cC}$ is a $4(2g+1)$-th power (resp. $2(2g+1)$-th
  power) if $g$ is odd (resp. even).  
  \end{enumerate}
  \end{corollary}

\begin{thm} \label{main} Let $\cC$ be a Weierstrass model of $C$ over
  $S$. Then $\cC$ is defined by an integral Weierstrass equation 
  (Definition~\ref{integral-eq}) under any of the following conditions: 
  \begin{enumerate}[{\rm (1)}]
  \item $\Delta_{\cC}$ is principal, 
    $\det f_*\omega_{\cC/S}$ is free and either
    \begin{enumerate}
    \item $g$ is odd and $\cZ\simeq \PP^1_S$ or, 
    \item $g\equiv 2 \mod 4$  
  \end{enumerate}
    (so the converse of
    Proposition~\ref{obvious} is true if $4 \nmid g$);
  \item $\Pic(S)$ is finite of odd order, 
    $\Delta_{\cC}$ is principal and
    $\det f_*\omega_{\cC/S}$ is free;
 \item $\Pic(S)$ is finite of order prime to $2(2g+1)$,
or $g$ is even and the torsion elements of $\Pic(S)$ are of order prime to
    $2(2g+1)$, and $\Delta_{\cC}$ is principal. 
  \item $\Pic(S)$ is finite of order prime to $2g$ and
    $\det f_*\omega_{\cC/S}$ is free. 
  \end{enumerate}
\end{thm}

\begin{proof} We have to show that $\cZ\simeq \PP^1_S$ and $[\mathfrak w]$ is trivial under
  any of the conditions in the theorem.
  \medskip
  
  (1.a) This follows from Lemma~\ref{a-b-F-N}(3). 
  \medskip

  (1.b) Take any section $D\in \cZ(S)$ and consider the fractional
  ideals
  $\mathfrak a, \mathfrak b$ as in Definition~\ref{def:ab}. Then
  Lemma~\ref{a-b-F-N}(1) implies that $\mathfrak a^{-g(g+1)/2}\mathfrak b^g\simeq \cO_S$, 
so $\mathfrak a$ is a square, 
and $\cZ\simeq \PP^1_S$. Lemma~\ref{a-b-F-N}(3) implies that $[\mathfrak w]=1$. 
  \medskip 
 
  (2) First we have $\cZ\simeq \PP^1_S$ by
  Proposition~\ref{prop:P1}. Lemma~\ref{a-b-F-N}(3) implies that
  $[\mathfrak w]^2=1$, hence $[\mathfrak w]=1$. 
  \medskip

  (3) If $\Pic(S)$ is finite of order prime to $2(2g+1)$, then 
  $\cZ\simeq \PP^1_S$. Therefore 
  $[\mathfrak w]^{4(2g+1)}=1$ by Lemma~\ref{a-b-F-N}(2). Thus $[\mathfrak w]=1$.

  Now suppose $g$ is even and the torsion elements of $\Pic(S)$ are of
  order prime to $2(2g+1)$. With the notation of Lemma~\ref{a-b-F-N}, 
  as $\Delta_{\cC}$ is principal, $\mathfrak a^{g+1}\mathfrak b^{-2}$ is a torsion element of order dividing $2(2g+1)$, therefore
  trivial. So $\mathfrak a$ is a square in $\Pic(S)$ as $g$ is even,
  and $\cZ\simeq \PP^1_S$. We can now argue as in the first part.   
  \medskip

  (4) Similar to (3). 
\end{proof}

As an immediate consequence, we proved a conjecture of M. Sadek \cite{Sad}. 

\begin{corollary} \label{conj-Sadek} 
  Let $K$ be a number field of class number prime to $2(2g+1)$. Let
  $C$ be a hyperelliptic curve of genus $g$ 
having everywhere   good reduction over $K$. Then the smooth model
$\cC$ of $C$ is defined by an integral Weierstrass equation. 
\end{corollary}

\begin{proof} By the uniqueness of the smooth projective model, the
  hyperelliptic involution of $C$ extends to an involution on $\cC$
  and the quotient is smooth because $\cC$ is a smooth curve. So
  $\cC$ is a Weierstrass model of $C$ with $\Delta_{\cC}=\cO_K$.
  By Theorem~\ref{main}(3), $\cC$ is defined by an integral Weierstrass
  equation. 
\end{proof}
\end{section}

\begin{section}{Pointed integral Weierstrass equations} \label{p-h}
  
The usual Weierstrass models of elliptic curves $E$ are {\em pointed} Weierstrass
models with respect to the origin $o$ of $E$, in the sense that $o$
extends to a section contained in the smooth locus of the Weierstrass
model. For hyperelliptic curves $C$ having a rational \emph{Weierstrass point} $P\in C(K)$ 
(that is $\dim_K H^0(C, \cO_C(2P))=2$), P. Lockhart (\cite{Lock}) defined and studied 
pointed integral Weierstrass equations for $(C, P)$.

\begin{definition} \label{def:p-w} A \emph{pointed Weierstrass model} of $(C, P)$ over $S$ 
is a Weierstrass model $\cC$ such that the Zariski closure
$\overline{\{ P\}}$ is contained in the smooth locus 
of $\cC \to S$. Such a model is said to be \emph{defined by a pointed
integral Weierstrass equation} of
$(C, P)$ if the affine scheme $\cC \setminus \overline{\{ P \}}$ is defined by
$$
y^2+Q(x)y=P(x)
$$
with $P(x), Q(x)\in \cO_S[x]$, $P(x)$ is monic of degree $2g+1$ and
$\deg Q(x)\le g$.
\smallskip

Note that it may happen that a pointed Weierstrass model $\cC$ is
defined by an integral Weierstrass equation, but not by a pointed
integral Weierstrass equation (see Proposition~\ref{examples-s}(e)). 
\smallskip

We also have an obvious notion of \emph{minimal pointed Weierstrass model}.
Such a model always exists and is unique up to isomorphism 
(Corollary~\ref{pmw-unique}). 
When $\cO_S$ is principal (e.g. a discrete
valuation ring), the minimal pointed Weierstrass model is defined by 
a pointed integral Weierstrass equation (see e.g., \cite{Lock}, Proposition 2.8). 
\end{definition}

Let $\cC$ be a pointed Weierstrass model of $(C, P)$. Let us recall
the definition of the \emph{pointed Weierstrass class} $[\mathfrak u] \in
\Pic(S)$ of $\cC$ similarly to \cite{Lock}, Definition 2.7.
Let 
\begin{equation}
  \label{eq:p-w-e}
   y^2+Q(x)y=P(x) 
\end{equation} 
be a pointed Weierstrass equation of $(C, P)$ over $K$,  
of discriminant $\Delta\in K^*$. For any closed point $s\in S$,
we have a pointed integral Weierstrass equation
$$ y_s^2+Q_s(x_s)y_s=P_s(x_s) $$
for $\cC\times_{S} \Spec \cO_{S,s}$, 
with discriminant 
$\Delta_s\in \cO_{S,s}\setminus \{ 0\}$. Using Lemma~\ref{change-eq},
we see that translating, if necessary,
$x$ and $x_s$ by elements of $K$ and $\cO_{S,s}$ respectively
and $y, y_s$ by polynomials of degree $\le g$  in $K[x]$ and in $\cO_{S,s}[x]$
respectively (which keeps the equations as pointed Weierstrass
equations), we can suppose that
$$x=a_sx_s, \quad y=b_s y_s, \quad a_s, b_s\in K^*.$$ 
 As $P(x)$ and $P_s(x_s)$ are monic of
degree $2g+1$ we find $a_s^{2g+1}=b_s^{2}$. Denote by
$u_s=a_s^{g}/b_s$, then $a_s=u_s^{-2}$ and $b_s=u_s^{-(2g+1)}$. 

\begin{definition}[\cite{Lock}, Definition  2.7] \label{p-w-c} 
  The class in $\Pic(S)$ of the fractional ideal $\mathfrak u$ of $K$ such that
  $\mathfrak u\otimes \cO_{S,s}=u_s\cO_{S,s}$ for all $s\in S$ is called
  the \emph{pointed Weierstrass class\footnote{In fact our definition is the
 inverse of that of \cite{Lock}.} of $\cC$}. It does not
  depend on the choice of the equation~\eqref{eq:p-w-e}. With the notation of
  Definition~\ref{def:ab}, if we take for $D$ the image of
  $\overline{\{ P\}}$ in $\cZ(S)$, 
  we have $\mathfrak a=\mathfrak u^{-2}$ and
  $\mathfrak b=\mathfrak u^{-(2g+1)}$. 

Note that $[\mathfrak u]$ differs from the non-pointed Weierstrass class
(Definition~\ref{w-c}). Using Lemma~\ref{a-b-F-N}(1), we see that 
$[\mathfrak w]=[\mathfrak u]^{g}$ if $g$ is odd, and 
$[\mathfrak w]=[\mathfrak u]^{2g}$ if $g$ is even.
\end{definition}

\begin{proposition} \label{pointed-main}
   Let $\cC$ be a pointed Weierstrass model of $(C, P)$.
Let $D\in \cZ(S)$ be the image of $\overline{\{ P\}}\in \cC(S)$.
   \begin{enumerate}[\rm (1)] 
   \item We have
     $$ \Delta_{\cC} \simeq  \mathfrak u^{4g(2g+1)},
     \quad \det f_*\omega_{\cC/S} \simeq \mathfrak u^{-g^2}.$$
   \item The model $\cC$ is defined by a 
     pointed integral Weierstrass equation if and only if
     $\mathfrak u$ is principal. 
   \item The quotient $\cZ$ is isomorphic to $\PP^1_S$.
   \end{enumerate}
\end{proposition}

\begin{proof} (1) Follows from Lemma~\ref{a-b-F-N}. 
  See also \cite{Lock}, Equality (2.3) and Proposition 1.12. 

  (2) The only if part is clear. Suppose conversely that
  $\mathfrak u$ is principal. Then 
  $\mathfrak a$ is trivial and $D$ is 
  an elementary section. As $\mathfrak b$  is principal, $\cC$ is
  defined by an integral Weierstrass equation by 
  Step (1) in the proof of Proposition~\ref{global-L}.  

  (3) We use Lemma~\ref{a-b-F-N}. 
  As $\mathfrak a=\mathfrak u^{-2}$ is a square in $\Pic(S)$ and
  $\det \phi_*\cO_{\cZ}(D) \simeq \mathfrak a^{-1}$ by
  Lemma~\ref{a-b-F-N} (1), we find
  $\cZ\simeq \PP^1_S$ (Proposition~\ref{prop:P1}).
\end{proof}

\begin{corollary}
  If $\Pic(S)$ has finite order prime to $2g(2g+1)$ and if
  $\Delta_{\cC}$ is principal (e.g. if $\cC$ is smooth over $S$), then
  $\cC$ is defined by a pointed integral Weierstrass equation. 
\end{corollary}

The next proposition is a straightforward generalization of \cite{Sil-84},
Theorem 5.

\begin{proposition} \label{twist} Let $(C, P)$ be a pointed hyperelliptic curve over
  a number field $K$. Let $\mathfrak u$ be any fractional ideal of $K$. 
  Then $[\mathfrak u]$ is the pointed Weierstrass class of
  the minimal pointed Weierstrass model of some quadratic twist of
  $(C, P)$.
\end{proposition}

\begin{proof} The proof is the same as for \cite{Sil-84}, Theorem 5.
  The only point we have to check is that the minimal pointed
  Weierstrass model commutes with \'etale base changes in higher
  genus. See Proposition~\ref{minimal-bc}. 
  Note that, in general, this is false for non-pointed minimal Weierstrass
  models in higher genus  (\cite{L_Tr}, Proposition 4, p. 4595). 
\end{proof}

\begin{remark} A referee raised the natural
  question to study, for each class $[\mathfrak a]$ in ${\mathrm{Cl}}(K)$,
   the density of the hyperelliptic curves over $K$ having
  (pointed or not) Weierstrass class equal to $[\mathfrak a]$, 
  indicating that this question is solved for elliptic
curves in \cite{B}.  Recent related works can be found in \cite{C-S}.  
\end{remark}
\end{section}

\begin{section}{Complements on minimal pointed Weierstrass models} 

  We prove the uniqueness of the minimal pointed Weierstrass model
  (see also Remark~\ref{also}) and its compatibility with \'etale base changes.
  These properties are local on $S=\Spec R$.
  \medskip 

  Suppose first that $R$ is a discrete valuation ring with 
  uniformizing element $t$, residue field $k$ and normalized valuation
  $v$.  
  Let $\cC, \cC'$ be two distinct Weierstrass models of $C$ over $S$, with
  respective quotients $\cZ, \cZ'$ by the hyperelliptic involution. Up to actions of
  $\mathrm{Gl}_2(R)$, there are parameters $x, x'$ of $\cZ, \cZ'$
  respectively  and $d>0$ such that $x=t^d x'$. Let $p\in \cC_k$ be a
  rational point over the zero of $x$ in $\cZ_k$. 

  \begin{lemma}  \label{disc-cp} 
    If $\cC$ is regular at $p$, then $v(\Delta_{\cC}) <
    v(\Delta_{\cC'})$. 
  \end{lemma}

  \begin{proof} This is a special case of \cite{L_Tr}, Proposition 3
    (iii)-(iv), p. 4594. However, {\it op. cit}. is written under the
    hypothesis that $k$ is perfect. So we give an {\it ad hoc} proof 
    here.
    
    We suppose $\mathrm{char}(k)=2$. The other cases
    are similar and easier to deal with. For $0\le i\le d$, denote by
    $\cZ_i$ the smooth model of $\PP^1_K$ with parameter $x_i=x/t^i$ and
    by $\cC_i$ the normalization of $\cZ_i$ in $K(C)$. So $\cC=\cC_0$
    and $\cC'=\cC_d$. We are going to
    prove that $v(\Delta_{\cC_i})\le v(\Delta_{\cC_{i+1}})$ for any
      $0\le i\le d-1$ with strict inequality when $i=0$. 

    Let 
    $$y^2+(a_1+a_3x+...)y=t a_0 + a_2 x + a_4x^2+ ..., \quad a_i\in R$$
    be an affine equation of $\cC$ with $p$ corresponding to $x=y=0$. As
    $\cC$ is regular at $p$, at least one of $a_0, a_1, a_2$ is a unit in
    $R$.
Using \cite{L_Tr}, Lemme 2, p. 4582, we verify that the equation 
    $$ y^2+(a_1+ta_3x_1+...)y=t (a_0 + a_2 x_1 + ta_4 x_1^2+...).$$
    defines an affine normal scheme, hence it is an affine equation of
    $\cC_1$ and we have 
$$v(\Delta_{\cC_1})=v(\Delta_{\cC})+ 2(2g+1)(g+1)>v(\Delta_{\cC}).$$ 
Suppose that $d\ge 2$. Let $p_1\in \cC_1(k)$ be a point corresponding
to $x_1=0$. Then 
$\cC_1$ is regular at $p_1$ if either $a_0$ or $a_1\in R^*$ and we
can then continue with induction. 
If $a_0, a_1\in tR$ and $a_2\in R^*$, then the above equation becomes 
$$ y_1^2+(t^{-1}a_1+ta_3x_2+\cdots)y_1=t^{-1}a_0 + a_2 x_2 + t^2a_4x_2^2+  \cdots,$$
where $y_1=y/t$. Again this defines a normal scheme, hence $\cC_2$. 
We have $v(\Delta_{\cC_1})=v(\Delta_{\cC_2})+2(2g+1)(g-1)\ge v(\Delta_{\cC_2})$. 
As $\cC_2$ is smooth at $x_2=0$, we can also continue with induction
if $d\ge 3$.     
  \end{proof}

\begin{corollary} \label{pmw-unique} Let $C$ be a hyperelliptic curve over $K$ and let
  $P\in C(K)$ be a Weierstrass point. Then $(C, P)$ admits a unique 
  minimal pointed Weierstrass model of $(C, P)$ over $S$. 
\end{corollary}

\begin{proof} The existence is obvious when $S$ is local. In the 
general case $C$ has good reduction over a dense open subset $U$ of
$S$. Then it is enough to glue local minimal pointed Weierstrass
models 
over the points of $S\setminus U$ with the smooth model over $U$. 

Now we prove the uniqueness. One can suppose $S$ is local. Let 
$\cC, \cC'$ be two distinct pointed Weierstrass models of $(C, P)$. 
We keep the notation preceding Lemma~\ref{disc-cp}.
Let $p_0$ be the specialization of $P$ in $\cC_k$. This is a smooth
point by definition. If $p_0=p$, then
$v(\Delta_{\cC})< v(\Delta_{\cC'})$ by the same lemma. Otherwise,
$x(P)\notin tR$ (may be $\infty$). 
This implies that $(x')^{-1}(P)=t^dx(P)^{-1}\in tR$. Hence the 
specialization of $P$ in $\cZ'_k$ is the zero of $x'^{-1}=t^d x^{-1}$
and Lemma~\ref{disc-cp} implies that
$v(\Delta_{\cC'})<v(\Delta_{\cC})$. In other words, two distinct
pointed Weierstrass models always have distinct discriminant ideals.
Therefore the minimal one is unique. 
\end{proof}

\begin{remark} \label{also} 
This corollary also follows from \cite{Lock}, Remark after Definition 2.1
(but whose proof is left to readers). Note that the non-pointed
minimal Weierstrass model is not unique in general even when $C$ has a
rational Weierstrass point. 
\end{remark}

\begin{proposition}[\'Etale base change] \label{minimal-bc}
  Let $(C, P)$ be a pointed hyperelliptic curve over $K$. Let $\cC$ be the minimal 
  pointed Weierstrass model of $(C, P)$ over $S$. 
  Let $S'\to S$ be an \'etale morphism of affine Dedekind schemes and 
let $K'=K(S')$. 
  Then $\cC\times_S S'$ is the minimal
  pointed Weierstrass model of $(C_{K'}, P)$ over $S'$.   
\end{proposition}

\begin{proof} The property is local on $S'$, so we can suppose that $S$
  and $S'=\Spec R'$ are local.
The base change $\cC\times_S S'$ is a normal scheme because $S'/S$ is
\'etale. So $\cC\times_S S'$ is a pointed Weierstrass model of
$(C_{K'}, P)$. 
\medskip

(1) {\it First suppose that the residue extension of $S'/S$ is
  trivial}. Then $R$ is dense in $R'$. Let $\cC'$ be a pointed 
Weierstrass model over $S'$. We will show that it is defined over
$S$. Let $\cZ'$ be its quotient by the
hyperelliptic involution. Let $x'$ be a parameter of $\cZ'$ whose
pole divisor is the Zariski closure $\overline{\{ Q\}}\subset \cZ'$, where
$Q$ is the image of $P$ in $\PP^1_K(K)$. As $R$ is dense in $R'$, we can find 
$x\in K(\PP^1_K)$ such that $x'=x + a'x'+b'$ with $a', b'\in \m_{R'}$. 
Thus $R'[x']=R'[x]=R[x]\otimes_R R'$ and $\cZ'=\cZ \times_S S'$ for
some smooth model $\cZ$ of $\PP^1_K$. The normalization $\cC_0$ of
$\cZ$ in $K(C)$ is a pointed Weierstrass model of $(C, P)$ over $S$ and $\cC_0\times_S S'$ is equal
to $\cC'$ because the former is normal. Therefore
$$v(\Delta_{\cC'})=v(\Delta_{\cC_0})\ge v(\Delta_{\cC})=
v(\Delta_{\cC\times_S S'})$$
and $\cC\times_S S'$ is minimal. 

(2) In the general case, we can first replace $R$ with its
henselization (the first step shows that $\cC$ commutes with
henselization of $R$). The integral closure of $R$ in the Galois
closure of $K'/K$ is a discrete valuation ring unramified over $R$. So
we can suppose that $S'/S$ is Galois with Galois group $\Gamma$.
Let $\cC'$ be the minimal pointed Weierstrass model of $(C_{K'},
P)$
over $S'$. As $P$ is invariant by $\Gamma$ and the minimal pointed 
Weierstrass model is unique, $\Gamma$ acts on $\cC'$ (as $S$-scheme)
and on the quotient of $\cC'$ by the hyperelliptic involution. The quotient 
$\cC_0:=\cC'/\Gamma$ is a Weierstrass model of $C$ over 
$S$. As $S'/S$ is \'etale, $\cC_0\times_S S'$ is normal and
$\cC'\to \cC_0\times_S S'$ is an isomorphism. Hence $\cC$ is smooth
along the Zariski closure of $P$ and is a pointed Weierstrass model
of $(C, P)$. 
Moreover, $v(\Delta_{\cC'})=v(\Delta_{\cC_0})\ge v(\Delta_{\cC})$.
So $\cC'\simeq \cC\times_S S'$. 
\end{proof}

\end{section}

\begin{section}{Examples} \label{examples} 
  
We construct examples of  Weierstrass models not defined by integral
(or pointed integral)  Weierstrass equations, but with various nice
properties. We will work with number fields $K$.  
Let us summarize first the examples we will construct.

  \begin{proposition} \label{examples-s} There exist number fields
    $K$, hyperelliptic curves $C$ over $K$ with Weierstrass model
    $\cC$ over $S=\Spec\cO_K$, of any of the following types: 
    \begin{enumerate}[{\rm (a)}] 
\item 
$$\cZ\not\simeq \PP^1_S, \quad \Delta_{\cC}\simeq \cO_K, \quad 
\det f_*\omega_{\cC/S}\simeq \cO_S;$$ 
\item 
$$\cZ\simeq \PP^1_S, \quad \Delta_{\cC}\simeq\cO_K, \quad 
\det f_*\omega_{\cC/S}\not \simeq \cO_S;$$ 
\item 
  $$\cZ \simeq \PP^1_S, \quad \Delta_{\cC}\not\simeq \cO_K, \quad
  \det f_*\omega_{\cC/S}\simeq \cO_S;$$ 
(by Proposition~\ref{obvious}, in all the above situations, $\cC$ is not defined by an integral Weierstrass equation.) 
\item 
  $$\cZ\simeq \PP^1_S, \quad \Delta_\cC\simeq \cO_K, \quad
  \det f_*\omega_{\cC/S}\simeq \cO_S$$
  but $\cC$ is not defined by an integral Weierstrass equation; 
\item 
  $\cC$ is a minimal pointed Weierstrass model, defined by an integral
  Weierstrass equation but not by a pointed integral Weierstrass
  equation. 
    \end{enumerate}
  \end{proposition}

  \begin{remark} The genus $g$ of the curves $C$ appearing in our construction 
    is not arbitrary. For (a) and for (b) there are
    infinitely many possible $g$ (always odd).  If we relax the
    condition $\cZ\simeq \PP^1_S$ in (b),  we find all odd integers
    $g\ge 5$ except those of the form $2^n-1$.
    For (c), $g$ can be any integer different from $1, 2$ and $4$. 
    For (d), $g$ can be any integer divisible by $4$. Finally, 
    for (e), $g$ can be any integer $\ge 2$.
  \end{remark}
  
\begin{example} 
Fix $g_0\ge 1$. We start with a hyperelliptic curve $C_0$ of genus
$g_0\ge 1$ over a number field $K$ with everywhere good
reduction. Such a curve always exists by enlarging $K$ if
necessary.\footnote{For some special values of $g_0$, we can
give explicit examples as follows. 
Let $r\ge 3$ be an odd square-free integer. Consider the 
hyperelliptic curve $C_0$ over $\mathbb Q$ defined by $ y^2+y=x^r$.
If $\xi_r\in \mathbb C$ is a primitive $r$-th root of
the unity and if $m\in \Z$ satisfies $4m\equiv -1 \mod r^{(r+1)/2}$, then
$C_0$ has everywhere good reduction over $K=\Q[m^{1/r},
\sqrt{1-\xi_r}]$. The corresponding smooth model $\cC_0$ is defined by an
integral global Weierstrass equation, and $(1-\xi_r)^{-1}(x-m^{1/r})$
is a parameter of $\cZ_0$.} Let $\cC_0$ be its smooth projective model over $S=\Spec\cO_K$,
and let $\cZ_0$ be the corresponding projective bundle over $S$.
Extending $K$ if necessary, we can suppose that
$\cZ_0\simeq \PP^1_S$. 
Let $B\subset \cZ_0$ be the branch locus of $\cC_0\to \cZ_0$. This
is a closed subset finite over $S$. By \cite{GLL}, Theorem 7.2, there
exists an integral point $D_0$ in $\cZ_0 \setminus  B$. 
Applying again the theorem to $\cZ_0\setminus (B\cup D_0)$,
we find an integral point $D_\infty$ such that $B$, $D_0$ and
$D_\infty$ are pairwise disjoint. Extending $K$ if necessary, we can suppose that 
$\cC_0$ is defined by an integral Weierstrass equation
$$ y^2+Q(x)y=P(x), $$
and $D_0, D_\infty$ are respectively the zero and pole divisors of $x$. 
Write
$$4P(x)+Q(x)^2=a_0x^{2g_0+2}+\cdots+a_{2g_0+2} \in \cO_K[x].$$ 
As $(D_0\cup D_\infty)\cap B=\emptyset$, we have $a_0, a_{2g_0+2}\in
\cO_K^*$. 

Extending again $K$, 
we can suppose that the class group $\Pic(S)$ of $K$  contains a
maximal cyclic $2$-subgroup $H_2\subseteq \Pic(S)$ of order $\ge 4$
(see Lemma~\ref{2-tor} below). Let $[I]\in H_2$ (it will be chosen
more precisely later). We will use $[I]$ to
construct a hyperelliptic curve $C$ over $K$ together with a Weierstrass
model $\cC$ as double coverings of $\cC_0$. 
Let $d$ be the order of $[I]$. It is a power of $2$. 
As $I^d$ is principal, we 
identify $\PP(\cO_K\oplus I^d)$ with  $\cZ_0$, and the sections
corresponding to the projections of $\cO_K\oplus I^d\to \cO_K$ and
$\cO_K\oplus I^d\to I^d$ respectively to $D_0$ and $D_\infty$.
Consider now $\cZ:=\PP(\cO_K\oplus I)$ and the canonical morphism 
  $$\theta: \ \PP(\cO_K\oplus I)\to \PP(\cO_K \oplus I^d)=\cZ_0.$$
Locally on $S$ it corresponds to taking the $d$-th powers of the
homogeneous coordinates. Let 
$$\cC:= \cC_0\times_{\cZ_0} \cZ.$$ 
Let us study some properties of $\cC$. Let $e$ be a basis of $I^d$.
For any open subset $U\subseteq S$ such that $I|_U$ has a basis $e_U$, 
we have $e_U^d=\alpha_U^{-1} e$ with $\alpha_U\in \cO_S(U)^*$ a unit.  Then
$\cC_U$ is defined by the equation
$$
y^2+Q(\alpha_U t^{d})y=P(\alpha_U t^d).
$$
By Lemma~\ref{d-th-power} below, this equation defines a hyperelliptic
curve $C$ over $K$ of
odd genus
$g=d(g_0+1)-1$ and the discriminant ideal $\Delta_{\cC_U}$ defined by this equation satisfies
$$
\Delta_{\cC_U}=\pm (a_0a_{2g_0+2})^{d-1}\alpha_U^{(2g_0+2)((2g_0+2)d-1)}d^{2g+2}
\Delta^d_{(\cC_0)_U}=d^{2g+2}\cO_U. 
$$ 
Therefore $\Delta_{\cC}=d^{2g+2}\cO_K$ and $\cC\to S$ is smooth aways from
primes dividing $2$. At any prime $\mathfrak p$ of $\cO_K$ dividing $2$, as
$Q(x)$ is non-zero in $k(\mathfrak p)[x]$, the fiber $\cC_{\mathfrak p}$ is reduced, so
$\cC$ is normal (\cite{L_B}, Lemma 4.1.18) and is a Weierstrass model of $C$ over $S$. 

Let $D$ be the section of $\cZ$ corresponding to the projection
$\cO_K\oplus I\to I$. Then $\det \phi_*\cO_{\cZ}(D)=I$ and
$\theta(D)=D_\infty$. As $\cC$ is obtained by base change from
$\cC_0$, we have $N\simeq \cO_K$ (tensor product the exact sequence
\eqref{eq:L} corresponding to $\cC_0$ by $\cO_{\cZ}$).
By Lemma~\ref{a-b-F-N}(1), we have
$$\det f_*\omega_{\cC/S}\simeq  I^{g(g+1)/2}\simeq I^{d\times  \frac{(d(g_0+1)-1)(g_0+1)}{2}},$$
$[\mathfrak w]=[I]^{-d(g_0+1)/2}$ (and also $\Delta_{\cC}\simeq \cO_K$, but this is less precise then the above equality $\Delta_{\cC}=d^{2g+2}\cO_K$).
Therefore: 
\begin{enumerate}[{\rm (a)}]
\item if we take for $[I]$ a generator of
$H_2$ and if $g_0$ is odd, then $[I]$ is not a square, so 
$$\cZ\not\simeq \PP^1_S, \quad \Delta_{\cC}\simeq \cO_K, \quad 
\det f_*\omega_{\cC/S}\simeq \cO_S, \quad [\mathfrak w]=1;$$ 
\item if $[I]$ is the square of some element of 
$H_2$ and  if $g_0$ is even,  then 
$$\cZ\simeq \PP^1_S, \quad \Delta_{\cC}\simeq\cO_K, \quad 
\det f_*\omega_{\cC/S}\not \simeq \cO_S, \quad [\mathfrak w]\ne 1$$ 
\end{enumerate}
\end{example}

\begin{example} \label{ex-pointed}
  Let $g\ge 1$. By \cite{AC}, there exists a quadratic field $K$ such
  that $\mathrm{Cl}(K)$   contains a cyclic subgroup of order $4g^2$
  (see the beginning of the proof of Lemma~\ref{2-tor} below). Let $S=\Spec \cO_K$
  and let $[\mathfrak u]\in \Pic(S)$. By Proposition~\ref{twist}, there exists a
  pointed hyperelliptic curve $(C, P)$ over $K$ such that the pointed 
  Weierstrass class of the minimal pointed Weierstrass model $\cC$ is
  equal to $[\mathfrak u]$. 

  \begin{enumerate}[{\rm (a)}] \setcounter{enumi}{2}
  \item  If $g\not\equiv 0 \mod 4$, we choose
    $[\mathfrak u]$ of order $g^2$. Then $\cZ \simeq \PP^1_S$,
    $\Delta_{\cC}$ is not principal, and $\det f_*\omega_{\cC/S}$ is
    free (Proposition~\ref{pointed-main}(1)).  
\item  If $g\equiv 0 \mod 4$ and if we choose $[\mathfrak u]$ of order
  $4g$,   then $\cZ\simeq \PP^1_S$, $\Delta_\cC$ is principal and
  $\det f_*\omega_{\cC/S}$ is
  free. Let $[\mathfrak w]$ be the Weierstrass class of $\cC$. Then
  $[\mathfrak w]=[\mathfrak u]^{2g}$ (Definition~\ref{p-w-c}) is non-trivial. So $\cC$ is not defined by an
integral Weierstrass equation (Proposition~\ref{global-L}).
\item If $g\ge 2$, choose $[\mathfrak u]$ of order $g$. Then $[\mathfrak w]$ is trivial. 
As $\cZ\simeq\PP^1_S$, the model $\cC$ is defined by an integral Weierstrass
equation by Proposition~\ref{global-L}. But it is not defined by a 
pointed integral Weierstrass equation.
\end{enumerate}
\medskip

\noindent Note that in all the above examples 
except (c), $\Delta_{\cC}$ is principal but not equal to $\cO_K$ in
general. 
\end{example}

\begin{lemma}\label{2-tor} 
Let $K$ be a number field. Let $d\ge 2$. Then there exists an
extension $L/K$ (of degree at most $2$) such that the class group of
$L$ contains a cyclic subgroup of order $d$.
\end{lemma}

\begin{proof} For any $n\ge 1$, in the proof of \cite{AC}, Theorem 1,
explicit quadratic extensions $F/\Q$ whose class groups contain 
a cyclic subgroup $C_n$ of order $n$ are constructed. 
Fix such an $F$ with $n=d[K:\Q]$. Let $L=KF$. 
Using the norm map 
$\mathrm{Cl}(L)\to \mathrm{Cl}(F)$, we see that the kernel of the
canonical map $\mathrm{Cl}(F)\to \mathrm{Cl}(L)$ is annihilated by $[K : \Q]$, 
so the image of $C_n$ in $\mathrm{Cl}(L)$ has order divisible by $d$
and we are done.
\end{proof} 

\newcommand{\disc}{\mathrm{disc}}

\begin{lemma}\label{d-th-power} Let 
$$ E_0 : \quad y^2+Q(x)y =P(x) $$
be a Weierstrass equation of a hyperelliptic curve $C_0$ of genus $g_0$ over
some field $K$ with
$4P(x)+Q(x)^2=a_0x^{2g_0+2}+\dots + a_{2g_0+2}$ and 
$a_0a_{2g_0+2}\ne 0$ (the latter condition means that the covering $C_0\to \mathbb
P^1_K$ is unramified above the zero and pole of $x$). Let $\alpha\in
K^*$ and let $d\ge 1$ be an integer prime to the characteristic of $K$.
Then 
  $$ E_d : \quad y^2 + Q(\alpha x^d) y = P(\alpha x^d) $$
is a Weierstrass equation of a hyperelliptic curve of genus
$g=d(g_0+1)-1$, and the discriminants of the above equations satisfy the
relation 
$$
\Delta(E_d)=\pm
(a_0a_{2g_0+2})^{d-1}\alpha^{(2g_0+2)((2g_0+2)d-1)}d^{2g+2}
\Delta(E_0)^d. 
$$
\end{lemma}

\begin{proof} (See \cite{L_Tr}, \S 2, for the definition of $\Delta(E)$.) This amounts to proving the following fact: if
  $F(x)=a_0x^{n}+\cdots + a_n \in K[x]$ has degree $n$ and $\alpha\in K^*$, then
  $$\disc(F(\alpha x^d))= \pm
  (a_0a_n)^{d-1}\alpha^{n(nd-1)}d^{nd}\disc(F(x))^d.$$
First by substituting $\alpha^{1/d}x$ to $x$, we are reduced to the
case $\alpha=1$. A direct computation of $\disc(F(x^d))$ in terms of
the product of the differences of the roots of $F(x^d)$ in some 
algebraic closure of $K$ leads to the desired equality. 
\end{proof}

\end{section}


\begin{thebibliography}{12}

\bibitem{AC}  N. C. Ankeny and S. Chowla: {\it
    On the divisibility of the class number
    of quadratic fields}. Pacific J. Math. {\bf 5} (1955), 321-324.

\bibitem{B} E. Bekyel: {\it The density of elliptic curves having a
    global minimal Weierstrass equation},  J. Number Theory {\bf 109}
  (2004), 41-58.

\bibitem{C-S} J. E. Cremona, M. Sadek: {\it Local and global densities
  for Weierstrass models of elliptic curves}, 
{\tt arxiv:2003.08454 [math.NT]}.

\bibitem{GLL} O. Gabber, Q. Liu and D. Lorenzini: {\it 
Hypersurfaces in projective schemes and a moving lemma}, 
    Duke Math. Journal, {\bf 164} (2015), 1187-1270. 

\bibitem{EGA} A. Grothendieck and J. Dieudonn\'e:
    {\it \'El\'ements de g\'eom\'etrie alg\'ebrique}, II. 
Publ. Math. IH\'ES, {\bf 8} (1961).
  
\bibitem{L_Tr} Q. Liu: {\it Mod\`eles entiers de courbes
    hyperelliptiques sur un corps de valuation discr\`ete}.  
    Trans. of AMS, {\bf 348} (1996), 4577-4610.  
    
\bibitem{L_B} Q. Liu: {Algebraic Geometry and Arithmetic Curves}. 
Oxford Graduate Texts in Math., 6 (2002), 576 pages, Oxford University Press. New (corrected) edition in 2006. 

\bibitem{Lock} P. Lockhart: {\it On the discriminant of a
    hyperelliptic curve}. Trans. Amer. Math. Soc. {\bf 342} (1994), 729-752. 
  
\bibitem{LK} K. L{\o}nsted and S. Kleiman: {\it Basics on families of
    hyperelliptic curves}. Compositio Mathematica, {\bf 38} (1979), 83-111
  
\bibitem{Sad} M. Sadek: {\it Private communication} (2020). 

\bibitem{Sil-84} J. Silverman: {\it Weierstrass equations and the
    minimal discriminant of an elliptic curve}. Mathematika, {\bf 31} (1984), 245-251. 

\bibitem{Stroeker} R. J. Stroeker: {\it  
Reduction of elliptic curves over imaginary quadratic number fields}. 
Pacific J. Math., {\bf 108} (1983), 451-463. 
  
\bibitem{SP} Stacks Project Authors: {\it Stacks Project}, 
{\text{stacks.math.columbia.edu}}, (2021). 
  
\end{thebibliography}
\end{document}